\documentclass{amsart}
\usepackage{amsmath, amssymb, amsthm}

\usepackage{enumitem}

\usepackage[colorlinks=true, linkcolor=blue, citecolor=blue, urlcolor=blue]{hyperref}
\usepackage[capitalise, noabbrev]{cleveref}
\usepackage{float}
\usepackage{geometry}
\usepackage{titlesec}

\usepackage{tikz,dirtytalk}
\usepackage{amsmath}
\usepackage{caption}
\usepackage{hyperref}
\hypersetup{bookmarks=true,
    bookmarksnumbered=true, 
    colorlinks=true,     
    linkcolor=blue}
\usepackage{fancyhdr}
\usepackage{setspace}
\usepackage{csquotes}
\usepackage{graphicx}
\usepackage{xcolor}
\usepackage{hyperref}
\usepackage{cite}
\titleformat{\section}
  {\normalfont\fontsize{12}{15}\bfseries}{\thesection}{1em}{}

\fancypagestyle{plain}{
  \fancyhf{}
  \fancyfoot[C]{\thepage}

}

\numberwithin{equation}{section}

\theoremstyle{plain}
\newtheorem{theorem}{Theorem}[section]

\newtheorem{proposition}[theorem]{Proposition}
\newtheorem{lemma}[theorem]{Lemma}
\newtheorem{corollary}[theorem]{Corollary}

\theoremstyle{definition}
\newtheorem{notation}[theorem]{Notation}

\newtheorem{remark}[theorem]{Remark}

\newtheorem{definition}[theorem]{Definition}
\newtheorem{example}[theorem]{Example}

\DeclareMathOperator{\lcm}{lcm}
\DeclareMathOperator{\pd}{pd}
\DeclareMathOperator{\LCM}{LCM}
\DeclareMathOperator{\Taylor}{Taylor}
\DeclareMathOperator{\Scarf}{Scarf}
\newcommand{\E}{\mathcal{E}}
\newcommand{\G}{\mathcal{G}}
\newcommand{\M}{\mathcal{M}}
\newcommand{\T}{\mathcal{T}}
\newcommand{\Tq}{\T_q}
\newcommand{\LL}{\mathbb{L}}
\newcommand{\MM}{\mathbb{M}}
\newcommand{\NN}{\mathbb{N}}
\newcommand{\bm}{\mathbf{m}}
\newcommand{\st}{\colon}
\newcommand{\qand}{\quad \mbox{and} \quad}

\newcommand{\qwhere}{\quad \mbox{where} \quad}

\title{Simplicial Resolutions of the Quadratic Power of Monomial Ideals}

\author[S.~M.~Cooper]{Susan M. Cooper}
\address{Department of Mathematics\\
University of Manitoba\\
Winnipeg, MB\\
Canada R3T 2N2}
\email{susan.cooper@umanitoba.ca}

\author[S.~Faridi]{Sara Faridi}
\address{Department of Mathematics \& Statistics,
 Dalhousie University,
 6297 Castine Way,
 PO BOX 15000,
 Halifax, NS,
 Canada B3H 4R2
 } 
\email{faridi@dal.ca}

\author[H.~Mahmood]{Hasan Mahmood}
\address{Department of Mathematics \& Statistics,
 Dalhousie University,
 6297 Castine Way,
 PO BOX 15000,
 Halifax, NS,
 Canada B3H 4R2
 } 
\email{hasan.mahmood@dal.ca.}

\thanks{Cooper was supported by NSERC Discovery grant~2024-05444.  Faridi was supported by NSERC Discovery grant~2023-05929}

\date{}

\begin{document}
\begin{abstract}
Given {\it any} monomial ideal $ I $ minimally generated by $ q $ monomials, we define a simplicial complex $\MM_q^2$ that supports a resolution of $ I^2 $. We also define a subcomplex $\MM^2(I)$, which depends on the monomial generators of $I$ and also supports the resolution of $ I^2 $. As a byproduct, we obtain bounds on the projective dimension of the second power of any monomial ideal. We also establish bounds on the Betti numbers of $ I^2 $, which are significantly tighter than those determined by the Taylor resolution of $ I^2 $.  {{Moreover, we introduce the permutation ideal $\Tq$ which is generated by $q$ monomials.  For any monomial ideal $I$ with $q$ generators, we establish that $\beta(I^2) \leq \beta({{\Tq}^2})$.}} We show that the simplicial complex $\MM_q^2$ supports the minimal resolution of ${\Tq}^2$. In fact, $\MM_q^2$ is the Scarf complex of  ${\Tq}^2$.

\end{abstract}
\maketitle

\setlist{font=\normalfont}

\section{Introduction} 

Monomial ideals offer a rich setting for exploring the interplay between algebraic, combinatorial and topological structures. One such foundational interconnection was described in Taylor's thesis \cite{Taylor1966}, which was further extended and developed in \cite{Bayer1998a}, \cite{Bayer1998b}, \cite{Faridi2014},\cite{Lyubeznik1998}, and \cite{Uwe2009}, to name a few. These works collectively highlight how algebraic invariants of monomial ideals, such as Betti numbers and minimal free resolutions, can be studied through combinatorial and topological constructions like simplicial and  cellular complexes. 

 The Taylor resolution of any monomial ideal $ I $, minimally generated by $ q $ monomials, is a free resolution of $ I $ constructed using a $(q-1)$-simplex. This resolution is obtained by \enquote{homogenizing} the simplicial chain complex of the $(q-1)$-simplex. This simplex, with vertices labeled by the monomial generators and faces labeled by the least common multiples of those vertices, is called the Taylor complex of $ I $, denoted by Taylor$(I)$. Bayer and Sturmfels \cite{Bayer1998b} extended the idea of constructing a free resolution of a monomial ideal from the $(q-1)$-simplex to more general simplicial complexes on $q$ vertices, and provided a criterion for when such a homogenized chain complex would be a resolution of the ideal. If the homogenized simplicial chain complex of a simplicial complex $\Delta$ results in a free resolution of $I$, then $\Delta$ is said to \say{support} a free resolution of $I$. 
 
 The Taylor complex $\text{Taylor}(I)$ on $q$ vertices works for any monomial ideal generated by $q$ monomials and always supports a free resolution of such ideals. However, this simplex is often excessively large, resulting in a resolution for $I$ that is far from minimal. Consequently, the upper bounds on the Betti numbers provided by the Taylor resolution, i.e., $\beta_j(I) \leq \binom{q}{j+1}$, are often far from sharp. 
 This issue becomes more pronounced when considering the Taylor resolution of powers of $ I $. For example, if $ q \geq 2 $, the Taylor resolution of $ I^2 $ is never minimal due to the divisibility relations shown in \cref{lemma1}. In \cite{Cooper2021}, the authors introduced a simplicial complex $\LL_q^2$, a \enquote{prototype} for the Taylor complex of $I^2$ in the case where $I$ is generated by $q$ square-free monomials. This complex is a proper subcomplex of a $\binom{q}{2}$-simplex. They showed that the resolution of $ I^2 $ supported on $\LL_q^2$  is closer to a minimal resolution than the Taylor resolution. This not only provides sharper bounds on the Betti numbers but also is minimal for many ideals.

 However, there is a limitation to the complex $\LL_q^2$ as it works only for square-free monomial ideals.  Indeed, \cref{Lnotenough} provides an example where $\LL_q^2$ does not support a free resolution for the quadratic power of a monomial ideal. In this paper, we address this gap by introducing a new simplicial complex $\MM_q^2$ and a subcomplex $\MM^2(I)$, and show that both support the free resolution of $ I^2 $ for {\it any} monomial ideal generated by $ q $ monomials (\cref{mainthm}). This complex is also a proper subcomplex of $\Taylor(I^2)$, substantially improving the upper bounds on the Betti numbers of $ I^2 $ compared to those provided by $\Taylor(I^2)$.

We also provide further evidence to the generally acknowledged fact  that polarization, a widely used technique in the study of monomial ideals, may be insufficient when studying resolutions of powers of monomial ideals (see \cref{rem: polarization}). Polarizing powers of ideals can lead to the loss of several homological invariants, such as projective dimension, regularity, and Betti numbers.

 To make this paper self-contained, we have included a preliminaries section that recalls all the basic notions and concepts used here. In \cref{p:tree}, we show that $\MM_q^2$ is a simplicial tree, which then leads to the proof that  $\MM_q^2$ and $\MM^2(I)$ support free resolutions of $I^2$ when $I$ is an ideal generated by $q$ monomials.  \cref{coropd} provides a sharp upper bound on the projective dimension of $ I^2 $. \cref{cor1} establishes sharper general bounds on the Betti numbers of $ I^2 $, while \cref{cor2} provides even tighter bounds by using information about the generators of $ I^2 $.

 In \cref{s:4}, we introduce the permutation ideal $\Tq$, which is generated by $ q $ monomials. These ideals turn out to be useful in studying the quadratic powers of monomial ideals.  \cref{scarfthm}, our second main theorem, shows that the Scarf complex of the square of the permutation ideal coincides with $\MM_q^2$. On the other hand, \cref{corobettivector} shows that the Betti vector of the permutation ideal ${\Tq}^2$ serves as an upper bound for the Betti vector of the quadratic power of any monomial ideal.

\section{Preliminaries}\label{s:2}

In this section, we establish the basic definitions and notations that will be used throughout the paper. 
For a comprehensive overview of this field, interested readers may consult \cite{Orlik2007} and \cite{PeevaBook}, both of which offer detailed expositions of this rich beautiful area of study.
We begin by letting $ K $ be a field and $ S = K[x_1, x_2, \ldots, x_n] $. An element of $ S $ of the form $ {x_1}^{a_1} {x_2}^{a_2} \cdots {x_n}^{a_n} $, where $ a_1, a_2, \ldots, a_n \in \NN \cup \{0\} $, is called a monomial. The set of monomials in $S$ forms a $ K $-basis of $ S $. An ideal $ I $ of $ S $ is said to be a monomial ideal if it is generated by a set of monomials in $S$. By the Hilbert Basis Theorem \cite{BinomialIdealsHH2018}, it is clear that any monomial ideal of $ S $ is generated by finitely many monomials. The minimal generating set of $I$ is unique and is denoted by $\G(I)$. 

A simplicial complex is a fundamental structure in algebraic topology and combinatorial mathematics. It also plays a vital role in this paper.  We introduce the concept of simplicial complexes and several associated terms.

\begin{definition}
Let $V$ be a vertex set consisting of $n$ elements. A simplicial complex $\Delta$ on $V$ is a collection of subsets of $V$ that is closed under inclusion. That is,  if $F \in \Delta$ and $G \subseteq F$, then $G \in \Delta$.
\begin{enumerate}
[label=(\alph*)]
    \item A member of $\Delta$ is referred to as a \textbf{face}.
    \item  The \textbf{facets} of $\Delta$ are those faces which are not contained in any larger face.
    \item The \textbf{dimension} of a face $F \in \Delta$ is given by $\dim(F) = |F| - 1$.
    \item The \textbf{dimension} of $\Delta$ is the largest dimension of any of its faces.
    \item $\Delta$ is known as a \textbf{simplex} if it contains exactly one facet.
    \item The \textbf{$f$-vector} $\mathbf{f}(\Delta) = (f_0, \ldots, f_d)$ of a $d$-dimensional simplicial complex $\Delta$ is a $(d+1)-\text{tuple}$ of integers, where $f_i$ represents the number of $i$-dimensional faces in $\Delta$.
\end{enumerate}
\end{definition}

A simplicial complex can be uniquely determined by its facets. We denote a simplicial complex $\Delta$ with facets $F_1, \ldots, F_q$ as
$\Delta = \langle F_1, \ldots, F_q \rangle.$ 
Moreover, low-dimensional simplicial complexes can be represented geometrically. \cref{fig:geom-complex} illustrates the geometric representation of a $3$-dimensional simplicial complex on $9$ vertices. 
\begin{figure}[H]
    \centering
    \begin{tikzpicture}[scale=1.5]

        \coordinate (A) at (3.5, 0.5);
        \coordinate (B) at (3.7, -0.2);
        \coordinate (C) at (4.4, -0.4);
        \coordinate (D) at (4, 1.3);
        \coordinate (E) at (2.2, 1.5);
        \coordinate (F) at (1.8, 0.8);
        \coordinate (H) at (1.5, 0.3);
        \coordinate (G) at (2.2, -0.3);
        \coordinate (I) at (0.7, -0.2);

        \filldraw[fill=gray, opacity=0.6] (B) -- (C) -- (D) -- cycle;
        \filldraw[fill=gray, opacity=0.6] (B) -- (D) -- (F) -- cycle;\filldraw[fill=gray, opacity=0.6] (A) -- (B) -- (D) -- cycle;
        \filldraw[fill=gray, opacity=0.6] (E) -- (D) -- (F) -- cycle;
        \filldraw[fill=gray, opacity=0.6] (F) -- (G) -- (H) -- cycle;
        \filldraw[fill=gray,opacity=0.6] (G) -- (H) -- (I) -- cycle;\filldraw[fill=red!60, opacity=0.6] (E) -- (F) -- (H) -- cycle;

        \draw[thick] (A) -- (B) -- (C) -- (D) -- cycle;
        \draw[thick] (F) -- (G) -- (H) -- cycle;
        \draw[thick] (A) -- (D);
        \draw[thick] (B) -- (D);
        \draw[thick] (E) -- (D);
        \draw[thick] (F) -- (D);
        \draw[thick] (F) -- (B);
        \draw[thick] (F) -- (G);
        \draw[thick] (E) -- (F) -- (H) -- cycle;
        \draw[thick] (H) -- (G) -- (I) -- cycle;
        \draw[dashed, thick, black!80!black] (A) -- (C);
        \draw[dashed, thick, black] (F) -- (D);
        \draw[dashed, thick, black!80!black] (A) -- (F);

        \foreach \p in {A,B,C,D,E,F,G,H,I}
          \fill[black] (\p) circle (1pt);

        \node[above left] at (A) {3};
        \node[below right] at (B) {2};
        \node[below right] at (C) {1};
        \node[above] at (D) {4};
        \node[above] at (E) {5};
        \node[above left] at (F) {6};
        \node[below left] at (G) {9};
        \node[above left] at (H) {7};
        \node[left] at (I) {8};

    \end{tikzpicture}
    \caption{A $3$-dimensional simplicial complex $\Delta$.}
    \label{fig:geom-complex}
\end{figure}
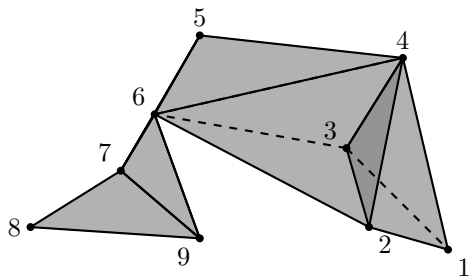

 We now aim to define a special type of simplicial complex, namely, a \say{simplicial tree}. This type of simplicial complex is of fundamental interest in our paper. 
 
\begin{definition}
Let $\Delta$ be a simplicial complex on a vertex set $V$. 
\begin{enumerate}[label=(\alph*)]
 
\item A \textbf{subcomplex} of $\Delta$ is a subset of $\Delta$ that is also a simplicial complex.

\item A {\bf subcollection} of $\Delta$ is a subcomplex of $\Delta$ whose facets are also facets of $\Delta$. 

\item  Given $W \subseteq V$, the \textbf{induced subcomplex} of $\Delta$ on $W$ is the subcomplex, denoted by $\Delta_W$, such that $\Delta_W = \{\sigma \in \Delta \mid \sigma \subseteq W\}$.

\item $\Delta$ is said to be \textbf{connected} if for any two facets $F$ and $G$, there is sequence of facets $F_0=F,F_1,F_2,\ldots,F_r=G$ such that 
$$ F_0 \cap F_1\neq \emptyset, \quad 
F_1\cap F_2\neq \emptyset, \quad 
\ldots, \qand
F_{r-1}\cap F_r\neq \emptyset.$$

\item A facet $F$ of $\Delta$ is a \textbf{leaf} if it is the only facet of $\Delta$, or there exists another facet $G \ne F$ (called a \textbf{joint}) such that $F \cap H \subseteq G$ for all facets $H \ne F$.

\item $\Delta$ is a \textbf{simplicial forest} if every subcollection of $\Delta$ has a leaf. A connected simplicial forest is said to be a \textbf{simplicial tree}.
\end{enumerate}
\end{definition}

 The left-most simplicial complex depicted in \cref{Figureqtree} is a simplicial tree on six vertices, while the one on the right is not, since the subcollection
 $\langle F_1,F_2,F_3\rangle$ has no leaf. Also, note that the complex on the right in \cref{Figureqtree} is isomorphic to $\LL_3^2$ defined in \cite{Cooper2021}.

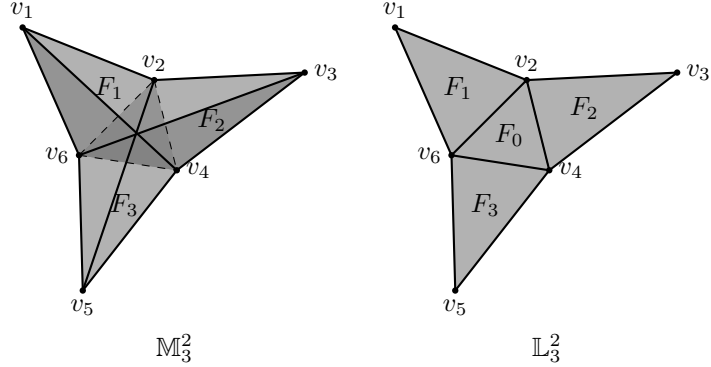
\begin{figure}[h]
    \centering
\begin{tabular}{cc}
    \begin{tikzpicture}[scale=1]

        \coordinate (A) at (4.5,0);  
        \coordinate (B) at (3.75,1.7);  
        \coordinate (C) at (5.5,1);  
        \coordinate (D) at (5.8,-0.2);  
        \coordinate (E) at (4.55,-1.8);  
        \coordinate (F) at (7.5,1.1);  
        
        \fill[gray,opacity=0.6] (E) -- (C) -- (D) -- cycle;
        \fill[gray,opacity=0.6] (A) -- (C) -- (E) -- cycle;
        
        \fill[gray,opacity=0.6] (A) -- (C) -- (D) -- cycle;
        \fill[gray,opacity=0.6] (A) -- (D) -- (B) -- cycle;
        \fill[gray,opacity=0.6] (A) -- (C) -- (B) -- cycle;
         \fill[gray,opacity=0.6] (F) -- (C) -- (D) -- cycle;
        \fill[gray,opacity=0.6] (F) -- (A) -- (D) -- cycle;
       
        \draw[thick] (A) -- (B);
        \draw[thick] (C) -- (B);
        \draw[thick] (D) -- (B);
        \draw[thick] (A) -- (E);
        \draw[thick] (C) -- (E);
        \draw[thick] (D) -- (E);
         \draw[thick] (A) -- (F);
        \draw[thick] (C) -- (F);
        \draw[thick] (D) -- (F);

        \draw[dashed] (A) -- (C);
        \draw[dashed] (A) -- (D);
         \draw[dashed] (C) -- (D);
        
        \foreach \point in {A,B,C,D,E,F} \fill[black] (\point) circle (1.2pt);

        \node[below] at (E) {$ v_5 $};
        \node[left] at (A) {$ v_6 $};
        \node[above] at (B) {$v_1$};
        \node[above] at (C) {$ v_2 $};
        \node[right] at (D) {$ v_4 $};
        \node[right] at (F) {$ v_3 $};
\node at ({4.9}, { (0 + 1.7 + 1)/3 }) {$F_1$};  
\node at ({(5.5 + 5.8 + 7.5)/3}, {0.47}) {$F_2$};  
\node at ({5.1}, { (0 + -0.2 + -1.8)/3 }) {$F_3$};  

\end{tikzpicture}
&
  \begin{tikzpicture}[scale=1]
        \coordinate (A) at (-0.5,0);  
        \coordinate (B) at (-1.25,1.7);  
        \coordinate (C) at (0.5,1);  
        \coordinate (D) at (0.8,-0.2);  
        \coordinate (E) at (-.45,-1.8);  
        \coordinate (F) at (2.5,1.1);  
        
      
        \fill[gray,opacity=0.6] (A) -- (B) -- (C) -- cycle;
        \fill[gray,opacity=0.6] (C) -- (D) -- (F) -- cycle;
        \fill[gray,opacity=0.6] (A) -- (D) -- (E) -- cycle;
         \fill[gray,opacity=0.6] (A) -- (C) -- (D) -- cycle;

        \draw[thick] (A) -- (B);
        \draw[thick] (A) -- (E);
        \draw[thick] (C) -- (B);
        \draw[thick] (C) -- (D);
        \draw[thick] (A) -- (C);
        
        \draw[thick] (A) -- (D);
        \draw[thick] (D) -- (E);
         
        \draw[thick] (C) -- (F);
        \draw[thick] (D) -- (F);

        
         \draw[dashed] (C) -- (D);
        
        \foreach \point in {A,B,C,D,E,F} \fill[black] (\point) circle (1.2pt);

        \node[below] at (E) {$ v_5 $};
        \node[left] at (A) {$ v_6 $};
        \node[above] at (B) {$v_1$};
        \node[above] at (C) {$ v_2 $};
        \node[right] at (D) {$ v_4 $};
        \node[right] at (F) {$ v_3 $};
       
        \node at ({(-0.5 + -1.25 + 0.5)/3}, { (0 + 1.7 + 1)/3 }) {$F_1$};  
        \node at ({(0.5 + 0.8 + 2.5)/3}, { (1 + -0.2 + 1.1)/3 }) {$F_2$};  
        \node at ({(-0.5 + 0.8 + -0.45)/3}, { (0 + -0.2 + -1.8)/3 }) {$F_3$};  
        \node at ({(-0.5 + 0.5 + 0.8)/3}, { (0 + 1 + -0.2)/3 }) {$F_0$};  

    \end{tikzpicture}\\
    $\MM^2_3$ & $\LL^2_3$
\end{tabular}
\caption{$\MM^2_3$ (on the left) is a simplicial tree with six vertices, whereas the complex $\LL_3^2$ (on the right) is not a simplicial tree.}

    \label{Figureqtree}
\end{figure}

 Free resolutions are fundamental tools for studying algebraic structures in commutative and homological algebra. Their history traces back to David Hilbert in the early twentieth century. In recent decades, modern approaches to studying free resolutions have become more combinatorial and geometric, e.g., see \cite{Eisenbud}, \cite{EzraMiller}, \cite{PeevaBook}. In this paper, we  aim to explore the free resolution of the second power $I^2$ for a monomial ideal $I$ through a topological approach. 
We now define free resolutions and introduce some related concepts.

\begin{definition}
Let $I$ be a monomial ideal in the polynomial ring $S = K[x_1, x_2, \ldots, x_n]$, minimally generated by $q$ monomials.  
A \textbf{free resolution} of $I$ is an exact sequence of the form  
\begin{equation}
    \cdots \to F_i \to F_{i-1} \to \cdots \to F_1 \to F_0 \to I \to 0, 
    \label{eqres}
    \end{equation}
where $i \in \mathbb{N}$ and each $S$-module $F_i$ is isomorphic to a finite direct sum of copies of $S$. A free resolution of $I$ is called a \textbf{minimal free resolution} if the rank of each $F_i$, denoted by $\beta_i$, is the smallest possible among all free resolutions of $I$ at the $i^\text{th}$ position. That is, it is of the form  
 \begin{equation}
    0 \to S^{\beta_p} \to S^{\beta_{p-1}} \to \cdots \to S^{\beta_1} \to S^{\beta_0} \to I \to 0. 
    \label{eqminres}
    \end{equation}

   The minimal free resolution of an ideal is unique up to isomorphism. The integer $\beta_i$ in \eqref{eqminres}, for each $ i \in [p]=\{1,\ldots,p\} $, is called the \textbf{$i^\text{th}$ Betti number} of $I$. 
    The length $p$ of the minimal free resolution of  $I$ is called the \textbf{projective dimension} of $ I $ over $S$ and is denoted by $\pd_S(I)$.
\end{definition}

 In her thesis \cite{Taylor1966}, Taylor showed that a free resolution of a monomial ideal $ I $ in $ S $ with $ q $ minimal generators can be constructed using a $(q-1)$-simplex. This simplex is known as the Taylor complex of $ I $, denoted $\Taylor(I)$, and it supports a free resolution of $I$.  In this complex, the vertices are labeled with the $ q $ minimal monomial generators of $ I $, and each face is labeled with the least common multiple of the monomial labels of its vertices. Further study in this area by Bayer, Peeva, and Sturmfels \cite{Bayer1998a,Bayer1998b} showed that other simplicial complexes with $ q $ vertices and such labeled faces can also be used to derive a free resolution of $ I $ by homogenizing their simplicial chain complexes. For a detailed illustration of how homogenization works in the case of a $2$-simplex, see {\cite[Example 2.1]{Chau2025}}. The criteria for a simplicial complex to support a free resolution of a monomial ideal are described in \cref{thmbasic} below. To state this theorem, we first introduce some necessary notation.

 For a monomial ideal $I$ minimally generated by $m_1,\ldots,m_q$ and $A \subseteq [q]=\{1,\ldots,q\}$ we define 
 $$m_A=\lcm(m_a \st a \in A)$$ 
 and we let $\LCM(I)$ be the {\bf lcm lattice} of $I$, which is  the set
 $$\LCM(I)=\{ m_A \st A\subset [q] \}$$
  partially ordered by division. 
  
  If $\Delta$ is a simplicial complex on $q$ vertices, we label its  $0$-dimensional faces with minimal monomial generators of $I$ and assign any other face $\tau\in \Delta$ the label  $$m_\tau=\lcm\{m_i \st m_i\in \tau\}.$$
 For any monomial $\bm$, we use the notation $\Delta_{\bm}$ to denote the induced subcomplex of $\Delta$ consisting of those faces whose monomial labels divide $\bm$, that is, 
$$\Delta_{\bm}=\{\tau\in \Delta \st m_\tau  \mid  \bm\}.$$

\begin{theorem}[{\bf Criterion for supporting a free resolution}]\label{BPS1} \label{thmbasic} 
Let $\Delta$ be a simplicial complex on $q$ vertices, and let $I$ be an ideal minimally generated by $q$ monomials  in a polynomial ring over a field. Label the vertices of $\Delta$ with the $q$ monomial generators of $I$ and label each face $\tau \in \Delta$ with the monomial $m_\tau$.   Then
\begin{enumerate}
    \item $\Delta$ supports a free resolution of $I$ if and only if for every monomial $\bm$ in $\LCM(I)$, $\Delta_{\bm}$ is empty or acyclic {{\em(}see~{\em\cite{Bayer1998a}}{\em)}};
    \item if $\Delta$ is a simplicial tree, then $\Delta$ supports a free resolution of $I$ if and only if for every monomial $\bm$ in $\LCM(I)$, $\Delta_{\bm}$ is empty or connected {\em(}see~{\em\cite{Faridi2014}}{\em)}; 
    \item the free resolution of $I$ supported by $\Delta$  is minimal if and only if for every pair of faces $\tau_1, \tau_2\in\Delta$ with $\tau_1\subsetneq\tau_2$, we have $m_{\tau_1}\neq m_{\tau_2}$ {{\em(}see~{\em\cite{Bayer1998a}}{\em)}}.
\end{enumerate}

\end{theorem}

\begin{remark}\label{remarkmain} \cref{BPS1} establishes that the $f$-vector of a simplicial complex $\Delta$, which supports a resolution of a monomial ideal $I$, serves as an upper bound for the Betti vector of $I$. More precisely, for each $i \leq d = \dim(\Delta)$, we have
\[
\beta_i(I) \leq f_i, \qwhere \mathbf{f}(\Delta) = (f_0, \dots, f_d).
\]
Furthermore, if $\Delta$ provides a minimal free resolution of $I$, then equality holds in the above relation.
\end{remark}

\cref{BPS1} says that the minimality of a free resolution of an ideal depends on the condition that no two nested subfaces share the same monomial label. If we extend this to an extreme condition where no two faces share the same monomial label, then this leads to the following definition.

\begin{definition}[{\bf Scarf Complex}]\label{scarf} Let $I$  be a monomial ideal minimally generated by the $q$ monomials $m_1,m_2,\dots,m_q$. The \textbf{Scarf complex} of $ I $, denoted by Scarf$(I)$, is the simplicial complex
\[
\Scarf(I) = \left\{ \tau \in \Taylor(I) \mid m_{\tau} \neq m_{\tau_1} \text{ for all } \tau_1 \in \Taylor(I), \tau_1 \not = \tau \right\}.
\]

\end{definition}
  The ideal $ I $ is called a \textbf{Scarf ideal} if its Scarf complex, $ \Scarf(I) $, supports a free resolution of $ I $. The faces of the Scarf complex of $ I $ are referred to as the \textbf{Scarf faces} of $ I $. A free resolution of $ I $ supported on its Scarf complex is called a \textbf{Scarf resolution}, which, in fact, forms a minimal free resolution of $I$ (see \cite{Mermin2012} for details). The following lemma from \cite{Faridi2023} characterizes the Scarf faces of a monomial ideal and we will use it later in this paper.

\begin{lemma}{\rm ({\cite[Lemma 2.2]{Faridi2023}})} \label{lem:scarf}
Let $ I $ be a monomial ideal and let $ \gamma \in \Taylor(I) $. Then $ \gamma \in \Scarf(I) $ if and only if both of the following statements hold:
\begin{enumerate}
    \item $ m_{\gamma} \neq m_{\gamma \setminus \{v\}} $ for all vertices $ v \in \gamma $;
    \item $ m_{\gamma \cup \{v\}} \neq m_{\gamma} $ for all vertices $ v \in \Taylor(I) \setminus \gamma $.
\end{enumerate}

\end{lemma}

\section{A simplicial complex that supports a free resolution of $ I^2 $}\label{s:3}

The authors in~\cite{Cooper2021} gave a description of a simplicial complex $\LL^2_q$ much smaller than the Taylor complex on the same number of generators, which supports a free resolution of $I^2$ when $I$ is any  ideal minimally generated by $q$ 
square-free monomials. By doing this, they replaced the binomial bounds on the Betti numbers of $I^2$ with much smaller ones. Our work in this paper begins with an observation that if
the generators of $I$ are not square-free, then $\LL^2_q$ may  not be large enough to support a free resolution of $I^2$.

\begin{example}[{\bf An example where $\LL_q^2$ is not  large enough}]\label{Lnotenough}
     Consider the monomial ideal $ I $ in six variables as in \cref{remark1}, where $ I = (m_1, m_2, m_3) $ with $ m_1 = {x_1}^3{x_2}^2{x_3}^3x_4{x_5}^2x_6 $, $ m_2 = {x_1}^2{x_2}^3x_3{x_4}^3x_5{x_6}^2 $, and $ m_3 = x_1x_2{x_3}^2{x_4}^2{x_5}^3{x_6}^3 $. Applying \cref{thmbasic} to $\mathbf{m} = \lcm({m_1}^2, m_2m_3)$ makes the induced complex $(\LL_3^2)_{\mathbf{m}}$ disconnected, indicating that $\LL_3^2$ (see \cref{Figureqtree}) does not support a resolution of $ I^2 $.
\end{example}

In this section, we introduce  two new simplicial complexes specifically designed to support a free resolution of $I^2$, where $I$ is a given monomial ideal. The construction of this simplicial complex is similar but  more general  than the complexes $\LL_q^2$ and  $\LL^2(I)$ described in \cite{Cooper2021} as it will cover the quadratic power of  {\it any} monomial ideal rather than just the square-free ones. We start with the following definition.

\begin{definition}\label{M2complexdefinition} Consider the vertex set $ V = \{\ell_{i,j} \st 1 \leq i \leq j \leq q\} $. Define the set $$\M = \{\ell_{i,j} \st 1 \leq i < j \leq q\},$$ 
and set $\M_i = \M \cup \{\ell_{i,i}\} $ for each $ i $. We denote by $\MM_q^2 $ the simplicial complex generated by the following $ q $ facets: $\M_1, \M_2, \ldots, \M_q $. That is,
\[ \MM_q^2 = \langle \M_1, \M_2, \ldots, \M_q \rangle. \]
\end{definition}

 Note that $\MM_q^2$ is a pure simplicial complex with ${\binom{q}{2}} + q$ vertices and $q$ facets, each of dimension $\binom{q}{2}$. For $q=1$, this complex reduces to a single point. \cref{Figureqtree2} illustrates the simplicial complexes for $q=2$ and $q=3$.

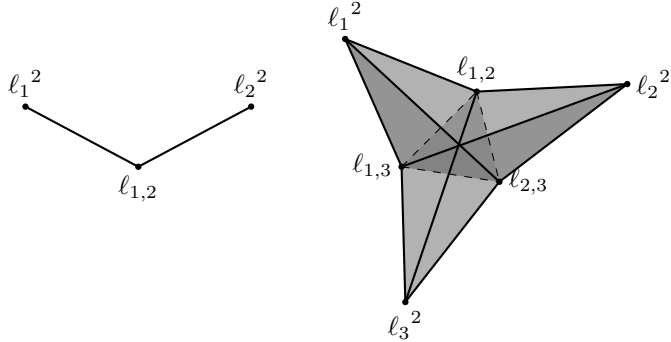
\begin{figure}[h]
    \centering
    \begin{tikzpicture}[scale=1]

        \coordinate (A) at (3.5,0);  
        \coordinate (B) at (2.75,1.7);  
        \coordinate (C) at (4.5,1);  
        \coordinate (D) at (4.8,-0.2);  
        \coordinate (E) at (3.55,-1.8);  
        \coordinate (F) at (6.5,1.1);  
        
        \fill[gray,opacity=0.6] (E) -- (C) -- (D) -- cycle;
        \fill[gray,opacity=0.6] (A) -- (C) -- (E) -- cycle;
        
        \fill[gray,opacity=0.6] (A) -- (C) -- (D) -- cycle;
        \fill[gray,opacity=0.6] (A) -- (D) -- (B) -- cycle;
        \fill[gray,opacity=0.6] (A) -- (C) -- (B) -- cycle;
         \fill[gray,opacity=0.6] (F) -- (C) -- (D) -- cycle;
        \fill[gray,opacity=0.6] (F) -- (A) -- (D) -- cycle;
       
        \draw[thick] (A) -- (B);
        \draw[thick] (C) -- (B);
        \draw[thick] (D) -- (B);
        \draw[thick] (A) -- (E);
        \draw[thick] (C) -- (E);
        \draw[thick] (D) -- (E);
         \draw[thick] (A) -- (F);
        \draw[thick] (C) -- (F);
        \draw[thick] (D) -- (F);

        \draw[dashed] (A) -- (C);
        \draw[dashed] (A) -- (D);
         \draw[dashed] (C) -- (D);
        
        \foreach \point in {A,B,C,D,E,F} \fill[black] (\point) circle (1.2pt);

        \node[below] at (E) {$ {\ell_{3}}^2 $};
        \node[left] at (A) {$ \ell_{1,3} $};
        \node[above] at (B) {${\ell_{1}}^2$};
        \node[above] at (C) {$ \ell_{1,2} $};
        \node[right] at (D) {$ \ell_{2,3} $};
        \node[right] at (F) {$ {\ell_{2}}^2 $};

        \coordinate (P) at (0,0);  
        \coordinate (Q) at (-1.5,0.8);  
        \coordinate (R) at (1.5,0.8);  
       
         \draw[thick] (P) -- (Q);
        \draw[thick] (P) -- (R);
        \foreach \point in {P,Q,R} \fill[black] (\point) circle (1.2pt);
        
        \node[below] at (P) {$ \ell_{1,2} $};
        \node[above] at (Q) {$ {\ell_1}^2 $};
        \node[above] at (R) {$ {\ell_2}^2 $};

    \end{tikzpicture}
    
    \caption{Simplicial complexes $ \MM_2^2 $ (on the left) and $ \MM_3^2 $ (on the right).}
    \label{Figureqtree2}
\end{figure}
  We want to show that the simplicial complex $\MM_q^2$ supports a resolution of $I^2$ for any monomial ideal generated by $q$ monomials. The idea is to show that $\MM_q^2$ is a simplicial tree and then invoke \cref{thmbasic}. In the next proposition, we show that this new simplicial complex is indeed a simplicial tree. 
  
\begin{proposition}\label{p:tree} For any $q\geq1$, the simplicial complex $\MM_q^2$ is a simplicial tree. 
\end{proposition}
\begin{proof}
Since $\MM^2_q$ is a connected collection of facets all joined in the same face, every facet of $\MM^2_q$ is a leaf, and any other facet is its joint. Every subcollection has the same form, so the result follows immediately.
\end{proof}


 If $I$  is a monomial ideal such that $\G(I) = \{m_1, m_2, \dots, m_q\}$, then the set $A = \{m_i m_j : i, j \in \G(I)\}$ generates $I^2$. However, this set may not be the minimal generating set of $I^2$, and $A$ may contain redundant generators as shown in the next example. 
 
 \begin{example}\label{example31} We consider the two ideals below of  $S=K[x_1,x_2,x_3,x_4]$.
\begin{enumerate}[label=(\alph*)]
    \item If  $I=(m_1,m_2,m_3)=({x_1}^2{x_2}^4{x_3}^3x_4, {x_1}^3{x_2}^3{x_3}^5{x_4}^2, {x_1}^4{x_2}^2{x_3}^3{x_4}^3)$, then  note that the generator $m_1m_3$ of $I^2$ divides ${m_2}^2$. This shows that ${m_2}^2$ is a redundant generator of $I^2$.
    \item Similarly, we see that for the ideal $$J=(m_1,m_2,m_3)=({x_1}^2{x_2}^4{x_3}^3x_4, {x_1}^3{x_2}^3{x_3}^2{x_4}^2, {x_1}^5{x_2}^2{x_3}^3{x_4}^3),$$  ${m_2}^2$ divides $m_1m_3$, making $m_1m_3$ redundant.
\end{enumerate}
\end{example}
 
 In such cases where the square of an ideal has redundant generators, the complex $\MM_q^2$ will provide an unnecessarily {large free resolution}, coming from the extra vertices corresponding to the  redundant generators. In order to avoid this, we now define a labeled induced subcomplex of $\MM_q^2$ by deleting some vertices that correspond to these redundant generators. 

\begin{definition}
 For an ideal $I$ minimally generated by the monomials $m_1, \ldots, m_q$, we define $\MM^2(I)$ to be a labeled induced subcomplex of $\MM^2_q$ formed by the following three rules.

\begin{enumerate}
    \item Label each vertex $\ell_{i,j}$ of $\MM^2_q$ with the monomial $m_i m_j$.
    \item For any indices $i, j, u, v \in [q]$ with $\{i,j\}\neq\{u,v\}$,  if we have $m_i m_j \mid m_u m_v$, then:
    \begin{itemize}
        \item if $m_i m_j = m_u m_v$ and $i = \min\{i, j, u, v\}$, then delete the vertex $\ell_{i,j}$;
        \item if $m_i m_j \ne m_u m_v$, then delete the vertex $\ell_{u,v}$.
      \end{itemize}
    \item Label each of the remaining faces with the least common multiple of the labels of its vertices.
\end{enumerate}

 The remaining labeled subcomplex of $\MM^2_q$ is called $\MM^2(I)$, and is a subcomplex of $\Taylor(I^2)$.
\end{definition}

For the two ideals  $I$ and $J$ of $S$ given in \cref{example31}, the simplicial complexes $\MM^2(I)$ and $\MM^2(J)$ are shown in \cref{Figureqtree3}.
\begin{figure}[h]
    \centering
    \begin{tikzpicture}[scale=1]

        \coordinate (A) at (0,0);  
        \coordinate (B) at (-1,2);  
        \coordinate (C) at (1,1);  
        \coordinate (D) at (1.2,-0.2);  
        \coordinate (E) at (-0.2,-2);  
        
        \fill[gray,opacity=0.6] (E) -- (C) -- (D) -- cycle;
        \fill[gray,opacity=0.6] (A) -- (C) -- (E) -- cycle;
        
        \fill[gray,opacity=0.6] (A) -- (C) -- (D) -- cycle;
        \fill[gray,opacity=0.6] (A) -- (D) -- (B) -- cycle;
        \fill[gray,opacity=0.6] (A) -- (C) -- (B) -- cycle;
        
        \draw[thick] (A) -- (B);
        \draw[thick] (C) -- (B);
        \draw[thick] (D) -- (B);
        \draw[thick] (A) -- (E);
        \draw[thick] (C) -- (E);
        \draw[thick] (D) -- (E);
        \draw[thick] (C) -- (D);

        \draw[dashed] (A) -- (C);
        \draw[dashed] (A) -- (D);
        
        \foreach \point in {A,B,C,D,E} \fill[black] (\point) circle (1.2pt);

        \node[below] at (E) {$ {m_3}^2 $};
        \node[left] at (A) {$ m_1m_3 $};
        \node[above] at (B) {$ {m_1}^2$};
        \node[right] at (C) {$ m_1m_2 $};
        \node[right] at (D) {$ m_2m_3 $};

        \coordinate (P) at (3.9,0);  
        \coordinate (Q) at (3.15,1.7);  
        \coordinate (R) at (4.9,1);  
        \coordinate (S) at (5.2,-0.2);  
        \coordinate (T) at (3.95,-1.8);  
        \coordinate (U) at (6.9,1.1);  
        
        
        \fill[gray,opacity=0.6] (R) -- (S) -- (Q) -- cycle;
        \fill[gray,opacity=0.6] (R) -- (S) -- (T) -- cycle;
         \fill[gray,opacity=0.6] (R) -- (S) -- (U) -- cycle;

        \draw[thick] (Q) -- (R);
        \draw[thick] (Q) -- (S);
        \draw[thick] (U) -- (R);
        \draw[thick] (U) -- (S);
        \draw[thick] (T) -- (R);
        \draw[thick] (T) -- (S);
        \draw[thick] (R) -- (S);
        


        \foreach \point in {Q,R,S,T,U} \fill[black] (\point) circle (1.2pt);

        \node[below] at (T) {$ {m_3}^2 $};
        \node[above] at (Q) {$ {m_1}^2$};
        \node[above] at (R) {$ m_1m_2 $};
        \node[right] at (S) {$ m_2m_3 $};
        \node[right] at (U) {$ {m_2}^2 $};

    \end{tikzpicture}
    
    \caption{For the ideals in \cref{example31}:  $\MM^2(I) $ (left) and $ \MM^2(J) $ (right).}
    \label{Figureqtree3}
\end{figure}
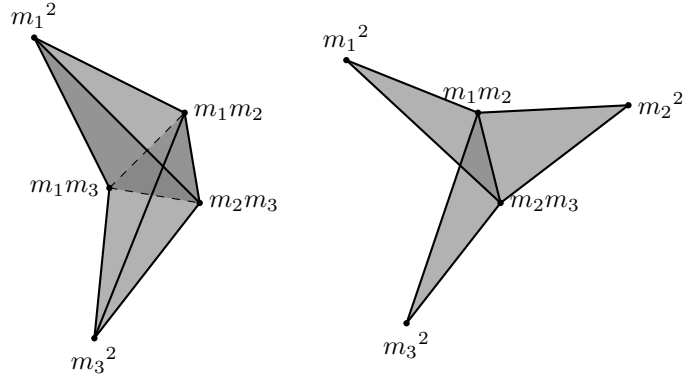

The next lemma will be helpful in proving our main result of this section, \cref{mainthm}.

\begin{lemma}\label{lemma1} If $I =  (m_1, \ldots, m_q)$ is a monomial ideal minimally generated by $q$ monomials $m_1, \ldots, m_q$ in $S = K[x_1, \ldots, x_n]$, then for any $i, j \in [q]$ we have the relation $m_i m_j$ divides $\mathrm{lcm}({m_i}^2, {m_j}^2)$.
\end{lemma}

\begin{proof}
    
For each $i$, let us write
\[ m_i = {x_1}^{a_{i1}} {x_2}^{a_{i2}} \cdots {x_n}^{a_{in}}, \]
where $a_{ik} \in \NN \cup \{0\}$, for all $i \in [q]$ and $k \in [n]$. With this notation, we have
\[ m_i m_j = {x_1}^{a_{i1} + a_{j1}} {x_2}^{a_{i2} + a_{j2}} \cdots {x_n}^{a_{in} + a_{jn}}, \]
\[ m_i^2 = {x_1}^{2a_{i1}} {x_2}^{2a_{i2}} \cdots {x_n}^{2a_{in}}, \]
\[ m_j^2 = {x_1}^{2a_{j1}} {x_2}^{2a_{j2}} \cdots {x_n}^{2a_{jn}}, \]
and
\[ \mathrm{lcm}({m_i}^2, {m_j}^2) = {x_1}^{2 \max \{a_{i1}, a_{j1}\}} {x_2}^{2 \max \{a_{i2}, a_{j2}\}} \cdots {x_n}^{2 \max \{a_{in}, a_{jn}\}}. \]

Now observe that for any $i, j \in [q]$ and $k \in [n]$,
\[ a_{ik} + a_{jk} \leq 2 \max \{a_{ik}, a_{jk}\}. \]

Thus, $m_i m_j \mid \mathrm{lcm}({m_i}^2, {m_j}^2)$.

\end{proof}

\begin{theorem}[{\bf Main Theorem 1}]\label{mainthm} 
Let $I\subset S$ be a monomial ideal minimally generated by $q$ monomials $m_1, m_2,  \ldots, m_q$. We have the following:
\begin{itemize}
    \item[(a)] $\MM_q^2$ supports a free resolution of $I^2$;
     \item[(b)] $\MM^2(I)$ supports a free resolution of $I^2$.
\end{itemize}
\end{theorem}

\begin{proof}
\noindent (a)
By \cref{p:tree}, $\MM_q^2$ is a simplicial tree. According to \cref{thmbasic}, to show that $\MM_q^2$ supports a free resolution of $ I^2 $, we must show that $(\MM_q^2)_{\bm}$ is either empty or connected for each monomial $ \bm \in \LCM(I^2) $, where  $(\MM_q^2)_{\bm}$ is the subcomplex of $\MM_q^2$ induced on the set $ V_{\bm} = \{\ell_{i,j} \in V : m_i m_j \mid  \bm\} $.

Now observe that if for some $i \ne j$ we have ${m_i}^2,{m_j}^2 \mid \bm$, then by \cref{lemma1} $m_im_j \mid \bm$, and so every time  $\ell_{i,i}, \ell_{j,j} \in (\MM_q^2)_{\bm}$, they will be both connected to the vertex $\ell_{i,j} \in (\MM_q^2)_{\bm}$. Moreover,  if for $a\neq b$ and $c \neq d$ we have $\ell_{a,b}, \ell_{c,d} \in (\MM_q^2)_{\bm}$, it means that $m_am_b, m_cm_d \mid \bm$, implying that $\lcm(m_am_b, m_cm_d) \mid \bm$. Therefore, once again, the edge connecting $\ell_{a,b}, \ell_{c,d}$ is in $(\MM_q^2)_{\bm}$. Therefore  $(\MM_q^2)_{\bm}$ is connected. 

\noindent (b) The simplicial complex $\MM^2(I)$, being an induced subcomplex of the simplicial tree $\MM_q^2$, forms a simplicial forest. Let $ V(\MM^2(I)) $ denote the set of vertices in $\MM^2(I)$. The proof of this part follows along the same lines as that of part (a) with the following modifications:

\begin{itemize}
    \item for any $ q \geq 2 $, replace $\MM_q^2$ with $\MM^2(I)$;
    \item for any $ i, j \in [q] $, replace $\ell_{i,j}$ with $ m_i m_j $;
    \item for any monomial ${\bm} \in \LCM(I^2)$, set $ V_{\bm} = \{m_i m_j \in V(\MM^2(I)) : m_i m_j \mid {\bm}\} $.
\end{itemize}

\end{proof}
 In view of \cref{remarkmain}, a direct consequence of \cref{mainthm} is \cref{coropd} given below.

\begin{corollary}\label{coropd}  If $I$ is  any monomial ideal minimally generated by $q$ monomials, then \begin{equation}\label{pdm}\pd_S(I^2)\leq \dim(\MM_q^2)={q\choose 2}.\end{equation}
    
\end{corollary}

\begin{remark}[{\bf The bound is sharp}]\label{remark1} Note that there are monomial ideals $I$ with $q$ generators  for which $\MM^2(I)=\MM_q^2$ and the free resolution supported on them is minimal.  For example, consider $S=K[x_1,x_2,x_3, x_4, x_5, x_6]$ and $I= ( {x_1}^3{x_2}^2{x_3}^3x_4{x_5}^2x_6, {x_1}^2{x_2}^3{x_3}{x_4}^3{x_5}{x_6}^2,  x_1x_2{x_3}^2{x_4}^2{x_5}^3{x_6}^3).$ Then $\MM_3^2$ supports the minimal free resolution of  $I^2$ which is minimally generated by the monomials: $${x_1}^6 {x_2}^4 {x_3}^6 {x_4}^2 {x_5}^4 {x_6}^2, \ {x_1}^5 {x_2}^5 {x_3}^4 {x_4}^4 {x_5}^3 {x_6}^3,  \ {x_1}^4 {x_2}^6 {x_3}^2 {x_4}^6 {x_5}^2 {x_6}^4, $$ $$  {x_1}^4 {x_2}^3 {x_3}^5 {x_4}^3 {x_5}^5 {x_6}^4,
 \ {x_1}^3 {x_2}^4 {x_3}^3 {x_4}^5 {x_5}^4 {x_6}^5,   \ {x_1}^2 {x_2}^2 {x_3}^4 {x_4}^4 {x_5}^6 {x_6}^6.$$
\end{remark}
 
 \begin{remark}[{\bf Comparison with square-free case}]\label{remark2}
The upper bound on the projective dimension of the quadratic power of a square-free monomial ideal $ I $ follows from Theorem 3.9 of \cite{Cooper2021}, and is given by
\begin{equation}\label{pds}
\operatorname{pd}_S(I^2) \leq \dim\left(\LL_q^2\right) = \binom{q}{2} - 1.
\end{equation}

There exist square-free monomial ideals, such as the extremal ideal $ \E_q $ defined in \cite{Cooper2024}, whose quadratic power attains the bound described in \eqref{pds}. Moreover, certain monomial ideals, such as the permutation ideal $ \Tq $ defined in \cref{s:4} of this paper, also achieve the upper bound on the projective dimension as described in \eqref{pdm}. For example, the quadratic power of the monomial ideal discussed in \cref{remark1} attains the bound in \eqref{pdm}, which equals 3 in the case where $ q = 3 $.
\end{remark}

\begin{remark}[{\bf Polarization and powers}]\label{rem: polarization}
When studying the free resolution of a general monomial ideal $ I $, polarization can be used to reduce the problem to the square-free case, as many homological invariants such as projective dimension, regularity, and graded Betti numbers are preserved under polarization; 
  see, for example, {\cite[Corollary 1.6.3]{HerzogMonomials}}. However, equations \eqref{pdm} and \eqref{pds} highlight that polarization is not an effective technique when studying the free resolution of powers of general monomial ideals, as it loses information on homological invariants like projective dimension. More specifically, if $I$ is the ideal in \cref{remark1}, and $\mathcal{P}(I)$ is the polarization of $I$,  then computations from Macaulay2 \cite{M2} show that
  $$\pd(I^2)=\pd(\mathcal{P}(I^2))=\binom{3}{2}=3 \qand \pd(\mathcal{P}(I)^2)=2 \lneq 3.$$
  Consequently, this results in a potential loss of information on regularity and graded Betti numbers as well. 
\end{remark}


 The simplicial complexes $\MM_q^2$ and $\MM^2(I)$ discussed in \cref{mainthm} provide tighter bounds on the Betti numbers of $I^2$ compared to those given by $\Taylor(I^2)$. The complex $\MM_q^2$ offers bounds based solely on the number of generators of the ideal. Whereas, the complex $\MM^2(I)$ provides upper bounds that depend on the specific generators of $I^2$. These bounds are given in the following corollaries.

\begin{corollary}\label{cor1} If $I$ is a monomial ideal of $S$ minimally generated by $q \geq 2$ monomials, then, for each $d \geq 0$, the $d^{th}$ Betti number $\beta_d({I}^2)$ satisfies the inequality
$$
\beta_d(I^2) \leq \binom{\frac{q(q-1)}{2}}{d+1}+q\binom{\frac{q(q-1)}{2}}{d}.
$$ 
    
\end{corollary}

\begin{proof}
By \cref{mainthm}(a), $I^2$ has a free resolution supported on $\MM_q^2$. Therefore, $\beta_d(I^2)$ is bounded above by $f_d(\MM_q^2)$ for any $d \geq 0$, where $f_d(\MM_q^2)$ is the $d^{th}$ component of the $f$-vector of $\MM_q^2$.

There are two types of $d$-dimensional faces in $\MM_q^2$: those that contain $\ell_{k,k}$ for some $k$, and those that do not contain any elements of the form $\ell_{k,k}$. The former type has $q \binom{\frac{q(q-1)}{2}}{d}$ faces, while the latter type has $\binom{\frac{q(q-1)}{2}}{d+1}$ faces. Consequently, we have
\[
\beta_d(I^2) \leq f_d(\MM_q^2) = \binom{\frac{q^2-q}{2}}{d+1} + q \binom{\frac{q^2-q}{2}}{d}.
\]
\end{proof}

 To show how \cref{cor1} refines this bound over the $9$-simplex for $\Taylor(I^2)$, we have included \cref{table1} for $q=4$.
\begin{table}[h]
    \centering
    \renewcommand{\arraystretch}{1.5} 
    \begin{tabular}{|c|c|c|c|c|c|c|c|}
        \hline
        {$d$} & 0 & 1 & 2 & 3 & 4 & 5 & 6 \\
        \hline\hline
        $f_d\bigl(9-\text{simplex}\bigr) = \binom{10}{d+1}$ & 10 & 45 & 120 & 210 & 252 & 210 & 120 \\
        \hline
        $f_d\bigl(\MM_4^2\bigr) = \binom{6}{d+1} + 4 \binom{6}{d}$ & 10 & 39 & 80 & 95 & 66 & 25 & 4 \\
        \hline
    \end{tabular}
    \caption{Upper bound comparison for Betti numbers of $I^2$ for $q=4$.}
    \label{table1}
\end{table}

 The next corollary provides an even tighter upper bound for the Betti numbers of $I^2$, incorporating the influence of the generators of the ideal.

\begin{corollary}{\label{cor2}}
Let $I$ be a monomial ideal of $S$ minimally generated by $q \geq 2$ monomials. Then for each $d \geq 0$, the $d^{th}$ Betti number $\beta_d(I^2)$ satisfies the inequality
\[
\beta_d(I^2) \leq \binom{\frac{q^2-q-2s}{2}}{d+1} + (q-t)\binom{\frac{q^2-q-2s}{2}}{d},
\]
where $s$ is the number of vertices of the form $\ell_{i,j}$ (with $i \neq j$) that are deleted from $\MM_q^2$ in the process of forming $\MM^2(I)$, and $t$ is the number of vertices of the form $\ell_{k,k}$ that are deleted during the same process.

\end{corollary}

\begin{proof}
Using part (b) of \cref{mainthm}, we have $\beta_d(I^2) \leq f_d(\MM^2(I))$. The computation of $f_d(\MM^2(I))$ follows similarly to the previous corollary,  based on the segregation of $(d+1)$-subsets of $\MM^2(I)$ into those that contain $\ell_{k,k}$ for some $k$ and those that do not contain any elements of the form $\ell_{k,k}$.
\end{proof}

 Consider the monomial ideal \begin{equation}\label{IdealExample}
 I=(m_1,m_2,m_3,m_4) =({x_1}^5,{x_1}^2{x_2}^4{x_3}^3,{x_1}^3{x_2}^3{x_3}^5{x_4}^2,{x_1}^4{x_2}^2{x_3}^3{x_4}^3)
 \end{equation}in the polynomial ring $S=K[x_1,x_2,x_3,x_4]$. Note that $m_2m_3\mid m_3m_4$, $m_1m_2\mid {m_4}^2$, and $m_2m_4\mid {m_3}^2$.  Hence,
   $|\G(I^2)|=7$, and from \cref{cor2}, we have  $s=1$, and $t=2$. Macaulay2, \cite{M2}, gives the following Betti table for $I^2$: 
   
   \[
\begin{array}{c|c|c|c|c}
d & 0 & 1 & 2 & 3  \\
\hline
\beta_d(I^2) & 7 & 12 & 8 & 2 \\
\end{array}
\]

    The bounds for the quadratic power of this ideal $I$ are provided in \cref{table2extended} below.




\begin{table}[h]
    \centering
    \renewcommand{\arraystretch}{1.5} 
    \begin{tabular}{|c|c|c|c|c|c|c|c|}
        \hline
        {$d$} & 0 & 1 & 2 & 3 & 4 & 5 & 6 \\
        \hline\hline
        $f_d\bigl({\Taylor}(I^2)\bigr) = \binom{7}{d+1}$ & 7 & 21 & 35 & 35 & 21 & 7 & 1 \\
        \hline
        $f_d\bigl(\MM^2(I)\bigr) = \binom{5}{d+1} + 2 \binom{5}{d}$ & 7 & 20 & 30 & 25 & 11 & 2 & 0 \\
        \hline
    \end{tabular}
    \caption{Upper bound (generator based) comparison for the Betti numbers of $I^2$, where 
$I$ is the ideal defined in \eqref{IdealExample}.}
    \label{table2extended}
\end{table}

\section{Permutation ideals and quadratic powers}\label{s:4}

 Let $ q $ be a positive integer, and let $ S_q $ denote the group of all permutations of the set $[q] = \{1, 2, \dots, q\}$. For any permutation $\sigma = i_1 i_2 \cdots i_q \in S_q$, define $\sigma(j) = i_j$.  That is,  $\sigma(j)$ is the element in the $j^{th}$ position of the permutation $\sigma$, reading from left to right. For instance, if $q = 5$ and $\sigma = 35412$, we have 
\[
\sigma(1) = 3, \quad \sigma(2) = 5, \quad \sigma(3) = 4, \quad \sigma(4) = 1, \quad \sigma(5) = 2.
\]
In this section, we introduce a special class of monomial ideals, called permutation ideals. A permutation ideal, denoted as $\Tq$, is given by $q$ generators. These ideals are defined to play a similar role for general monomial ideals as the extremal ideals, see \cite{Cooper2024}, \cite{Faridi2023}, play in the study of powers of square-free monomial ideals. Permutation ideals are constructed such that for any monomial ideal $I$ minimally generated by $q$ monomials in the polynomial ring $S = K[x_1, x_2, \ldots, x_n]$, we have 
$$\beta(I^2)\leq\beta({{\Tq}^2}).$$
They are also important because the simplicial complex $\MM_q^2$ supports a minimal resolution of the square of ${\Tq}^2$. In fact, our main theorem of this section proves that $\MM_q^2$ is the Scarf complex of ${\Tq}^2$ (\cref{scarfthm}).

\begin{definition}[{\bf Permutation Ideal}]   Let $ q $ be a positive integer, and let $ S_q $ denote the  symmetric group of permutations of the set $[q]$. For each permutation $ \sigma \in S_q $, associate a variable $ x_{\sigma} $, and consider the polynomial ring $ S_\T = K[x_{\sigma} : \sigma \in S_q] $ over a field $ K $.
For each $ i \in [q] $, define a monomial $ \tau_i $ in $ S_\T $ as
\[
\tau_i = \prod\limits_{\sigma \in S_q} {x_{\sigma}}^{\sigma(i)},
\]
and define the permutation ideal
\[
\Tq = (\tau_1, \tau_2, \ldots, \tau_q).
\]
\end{definition}

 Note that $ S_\T $ is a polynomial ring in $q!$ variables over $K$. For each $i\in [q]$, $$\deg(\tau_i)=\sum\limits_{\sigma \in S_q}\sigma(i) =(q-1)!(1+2+3+\cdots+q)=\frac{(q+1)!}{2}. $$

\begin{example} For $q=4$, we have
\\ $ S_\T=K[x_{1234}, \, x_{1243}, \, x_{1324}, \, x_{1342}, \, x_{1423}, \, x_{1432}, 
x_{2134}, x_{2143}, \, x_{2314}, \, x_{2341}, \, x_{2413}, \, x_{2431}, 
x_{3124},\\ x_{3142}, \, x_{3214},\, x_{3241}, \, x_{3412}, \, x_{3421}, 
x_{4123},\, x_{4132},\, x_{4213}, \, x_{4231},\, x_{4312}, \, x_{4321}]$, {\text{and}} 
$$\tau_1=x_{1234}x_{1243}x_{1324}x_{1342} x_{1423}x_{1432}x_{2134}^2 x_{2143}^2 x_{2314}^2 x_{2341}^2x_{2413}^2x_{2431}^2 
x_{3124}^3 x_{3142}^3x_{3214}^3x_{3241}^3x_{3412}^3$$ $$ x_{3421}^3x_{4123}^4x_{4132}^4x_{4213}^4x_{4231}^4x_{4312}^4x_{4321}^4,$$
$$\tau_2=x_{1234}^2x_{1243}^2x_{1324}^3x_{1342}^3 x_{1423}^4x_{1432}^4x_{2134} x_{2143} x_{2314}^3 x_{2341}^3x_{2413}^4x_{2431}^4 
x_{3124} x_{3142}x_{3214}^2x_{3241}^2x_{3412}^4$$ $$ x_{3421}^4x_{4123}x_{4132}x_{4213}^2x_{4231}^2x_{4312}^3x_{4321}^3,$$
$$\tau_3=x_{1234}^3x_{1243}^4x_{1324}^2x_{1342}^4 x_{1423}^2x_{1432}^3x_{2134}^3 x_{2143}^4 x_{2314} x_{2341}^4x_{2413}x_{2431}^3 
x_{3124}^2 x_{3142}^4x_{3214}x_{3241}^4x_{3412}$$ $$ x_{3421}^2x_{4123}^2x_{4132}^3x_{4213}x_{4231}^3x_{4312}x_{4321}^2,$$
$$\tau_4=x_{1234}^4x_{1243}^3x_{1324}^4x_{1342}^2 x_{1423}^3x_{1432}^2x_{2134}^4 x_{2143}^3 x_{2314}^4 x_{2341}x_{2413}^3x_{2431} 
x_{3124}^4 x_{3142}^2x_{3214}^4x_{3241}x_{3412}^2$$ $$ x_{3421}x_{4123}^3x_{4132}^2x_{4213}^3x_{4231}x_{4312}^2x_{4321}.$$

\end{example}

 As $\Tq$ is a monomial ideal generated by $q$ monomials, by \cref{mainthm} the simplicial complex $\MM_q^2$ supports a free resolution of ${\Tq}^2$.  We now show that the free resolution of ${\Tq}^2$ supported on $\MM_q^2$ is also minimal by showing that $\Scarf({\Tq}^2) =\MM_q^2$.  We will use \cref{lem:scarf} to achieve this goal.  We first need to show that the ideal ${\Tq}^2$ has as many generators as the number of vertices of  $\MM_q^2$.

Let $\sigma \in S_q$ and $i,j \in [q]$.  Recalling that $\sigma$ is bijective, we observe that
\begin{equation}\label{eq:gen-ii}
{x_\sigma}^{2q} \mid \tau_i\tau_j \iff 
i=j \qand \sigma(i)=q.
\end{equation}
and 
\begin{equation}\label{eq:gen-ij}
{x_\sigma}^{2q} \nmid \tau_i\tau_j \qand {x_\sigma}^{2q-1} \mid \tau_i\tau_j \iff 
\{\sigma(i), \sigma(j)\} = \{q,q-1\}. 
\end{equation}

\begin{lemma}\label{samevertexlemma} 
Let $\Tq$ be the permutation ideal in $S_\T$. Then
\begin{itemize}
    \item[(a)] $\Tq$ is minimally generated by $\{\tau_1, \tau_2, \ldots, \tau_q\}$
    \item[(b)] ${\Tq}^2$ is minimally generated by $\{\tau_{i}\tau_{j} \st 1 \leq i \leq j \leq q\}$. In particular, $$|\G({\Tq}^2)| = \binom{q}{2} + q = |V(\MM_q^2)|.$$
\end{itemize}
\end{lemma}

\begin{proof}
The statement in~(a) follows from the fact that for any distinct $i, j \in [q]$, we can choose a permutation $\sigma \in S_q$ such that $\sigma(i) = q$ and $\sigma(j) = q - 1$, which implies that $\tau_i \nmid \tau_j$.  Therefore, no generator $\tau_k$ divides any other, and the set $\{\tau_1, \ldots, \tau_q\}$ is indeed the minimal generating set of $\T_q$.

 To see~(b), observe that the set 
$G = \{\tau_{i}\tau_{j} \st 1 \leq i \leq j \leq q\}$
is a generating set of $ {\Tq}^2 $. We need to show that there are no redundant generators of $ {\Tq}^2 $ in $ G $. Suppose, on the contrary, that there exist two distinct (multi)subsets $ \{i_1, j_1\} $ and $ \{i_2, j_2\} $ of $ [q] $ such that 
\begin{equation}\label{eq:div}
\tau_{i_1}\tau_{j_1} \mid \tau_{i_2}\tau_{j_2}.
\end{equation}

\noindent {\bf \underline{Case 1}:} If $i_1 \neq j_1$, pick $\sigma \in S_q$ such that
$\sigma(i_1)=q$ and $\sigma(j_1)=q-1$. Then by \eqref{eq:gen-ij} and \eqref{eq:div} 
$${x_\sigma}^{2q-1} \mid 
\tau_{i_1}\tau_{j_1} \mid 
\tau_{i_2}\tau_{j_2}.$$
By \eqref{eq:gen-ij}, we then have either $ \{i_1, j_1\} =\{i_2, j_2\}$, in which case we are done, or otherwise 
${x_\sigma}^{2q} \mid \tau_{i_2}\tau_{j_2}$, which by \eqref{eq:gen-ii} means that $i_2=j_2$, and $\sigma(i_2)=q$. As $\sigma$ is bijective, this means that 
$i_1=i_2=j_2$. Hence we have
$$\tau_{i_1}\tau_{j_1} \mid {\tau_{i_1}}^2
\implies \tau_{j_1} \mid \tau_{i_1}$$
contradicting (a).

 \noindent {\bf \underline{Case 2}:} If $i_1 = j_1$, pick $\sigma \in S_q$ such that
$\sigma(i_1)=q$. Then by \eqref{eq:gen-ii} and \eqref{eq:div}
$${x_\sigma}^{2q} \mid 
{\tau_{i_1}}^2 \mid 
\tau_{i_2}\tau_{j_2}$$
which by \eqref{eq:gen-ii} implies that
 $i_1=i_2=j_2$.  This contradicts the assumption that $\{i_1, j_1\}$ and $\{i_2, j_2\}$ are distinct (multi)sets. 
\end{proof}

\begin{example}For $q=3$, we have $ S_\T=K[x_{123}, \, x_{132}, \, x_{213}, \, x_{231}, \, x_{312}, \, x_{321}]$ and the ideal ${{\T_3}}^2$ has the following ${3\choose 2}+3$ generators in  the minimal generating set $\G({\T_3}^2)=\{{\tau_1}^2,\tau_1\tau_2,\tau_1\tau_3,{\tau_2}^2,\tau_2\tau_3, {\tau_3}^2\}$, where 

$$
\begin{array}{l @{\hspace{6em}} l}
{\tau_1}^2 = x_{123}^2x_{132}^2x_{213}^4x_{231}^4x_{312}^6x_{321}^6 
& \tau_1\tau_2 = x_{123}^3x_{132}^4x_{213}^3x_{231}^5x_{312}^4x_{321}^5\\
&\\
{\tau_2}^2 = x_{123}^4x_{132}^6x_{213}^2x_{231}^6x_{312}^2x_{321}^4 
& \tau_2\tau_3 = x_{123}^5x_{132}^5x_{213}^4x_{231}^4x_{312}^3x_{321}^3\\
&\\
{\tau_3}^2 = x_{123}^6x_{132}^4x_{213}^6x_{231}^2x_{312}^4x_{321}^2
& \tau_1\tau_3 = x_{123}^4x_{132}^3x_{213}^5x_{231}^3x_{312}^5x_{321}^4.
\end{array}
$$   
\end{example}

\begin{notation} \cref{samevertexlemma} allows us to assign a new labeling to the faces of $\MM_q^2$, which we will use for the remainder of this section. The labeling is as follows: we assign to each vertex $\ell_{i,j}$ of $\MM_q^2$ the corresponding monomial generator $\tau_i \tau_j$ of ${\Tq}^2$, and to each face $\gamma \in \MM_q^2$ the monomial label
\[
m_\gamma = \lcm\{m \mid m \in \gamma\}.
\]
With this labeling, the simplicial complex $\MM_q^2$ now consists of the facets  
$\M_1, \M_2, \ldots, \M_q$, where for each $i \in [q]$,
\[
\M_i = \M \cup \{{\tau_{i}}^2\},
\]
such that
\[
\M = \{\tau_{i} \tau_{j} \mid 1 \leq i < j \leq q\}.
\]
Moreover, if $A = \{\tau_{i_1} \tau_{j_1}, \tau_{i_2} \tau_{j_2}, \ldots, \tau_{i_s} \tau_{j_s}\}$ is any subset of $ V(\MM_q^2) $, then 

\begin{equation}\label{4.1}
m_A= \prod\limits_{\sigma \in S_q} x_{\sigma}^{\max \{\sigma(i_1)+\sigma(j_1),\  \sigma(i_2)+\sigma(j_2),\ \dots\ ,\ \sigma(i_s)+\sigma(j_s)\}}.
\end{equation}
    
\end{notation}

 We now establish a few more results that will be useful in proving the main result of this section and, more broadly, in the study of permutation ideals.

\begin{lemma}\label{iiilemma} With the notation set above, we have
\begin{enumerate}[label=(\roman*)]

    \item  $m_{\M}= \prod\limits_{\sigma \in S_q} {x_{\sigma}}^{2q-1};$

    \item for a face $\gamma \in \MM_q^2$ and a permutation $\sigma \in S_q$, 
    \begin{equation}\label{eq:sigma-gamma}
    {x_\sigma}^{2q} \mid m_\gamma \iff 
    {\tau_i}^2 \in \gamma\  \text{for some}\ i\in[q] \ \text{with}\ \sigma(i)=q;
    \end{equation}

    \item for a face $\gamma \in \MM_q^2$   and $\sigma \in S_q$ with $\sigma(i)=q$ and $\sigma(j)=q-1$, if ${\tau_i}^2 \notin \gamma$ then 
    \begin{equation}\label{eq:sigma-gamma-M}
    {x_\sigma}^{2q-1} \mid m_\gamma \iff 
    \tau_i \tau_j \in \gamma;
    \end{equation}
   
     \item for any $k\in[q]$, we have $m_{\M_{k}} = \prod\limits_{\sigma \in S_q, \ \sigma(k)=q} {x_{\sigma}}^{{2q}}\prod\limits_{\sigma \in S_q, \ \sigma(k)\neq q} {x_{\sigma}}^{{2q-1}};$ 
       
  \item $m_{V(\MM_q^2)}= \prod\limits_{\sigma \in S_q} {x_{\sigma}}^{2q}$.
\end{enumerate}
\end{lemma}

\begin{proof}
Using \eqref{4.1} and \eqref{eq:gen-ij}, to find $m_{\M}$ we observe that 
\begin{equation*}
m_{\M}= \prod\limits_{\sigma \in S_q} x_{\sigma}^{\max\{\sigma(i)+\sigma(j) \st 1 \ \leq \ i \ < \ j \ \leq \  q \}}
=\prod\limits_{\sigma \in S_q} {x_{\sigma}}^{2q-1}.
\end{equation*}

The statement in \eqref{eq:sigma-gamma} follows directly from 
\eqref{eq:gen-ii} and, similarly,  \eqref{eq:sigma-gamma-M} is a consequence of 
\eqref{eq:gen-ij}. 
The remaining statements now follow immediately.

\end{proof}

\begin{proposition}\label{thmdistinctlabels}
For any two distinct faces  $\gamma_1$ and $\gamma_2$ of $ \MM_q^2 $, we have $m_{\gamma_1}\neq m_{\gamma_2}$.
    \end{proposition}
    
\begin{proof} 
Let $\gamma_1$ and $\gamma_2$ be any two distinct faces of $ \MM_q^2$ and assume  $m_{\gamma_1} = m_{\gamma_2}$. By \eqref{eq:sigma-gamma}  if a vertex ${\tau_i}^2$ is in one of $\gamma_1$ or $\gamma_2$ and not in the other, we immediately get  $m_{\gamma_1} \neq  m_{\gamma_2}$. So we may assume they both contain the same vertices of the form 
${\tau_i}^2$. Since a facet of $\MM^2_q$ can have at most one such vertex, we may assume both $\gamma_1, \gamma_2 \subseteq \M_i$.

Suppose the vertex labeled $\tau_a \tau_b$
is in $\gamma_1$ but not in $\gamma_2$  for some $a \neq b$ and let $\sigma \in S_q$ be a permutation such that $\sigma(a) = q$ and $\sigma(b) = q - 1$. Then 
${x_\sigma}^{2q-1} \mid m_{\gamma_1}= m_{\gamma_2}$. By \eqref{eq:sigma-gamma-M} we must have ${\tau_a}^2 \in \gamma_2$, which implies that $a=i$.

Now let $\sigma'\in S_q$  be a permutation such that $\sigma'(b) = q$ and $\sigma'(a) = q - 1$. Then 
${x_{\sigma'}}^{2q-1} \mid m_{\gamma_1}= m_{\gamma_2}$. Thus by \eqref{eq:sigma-gamma-M} we must have ${\tau_b}^2 \in \gamma_2$, which implies that $b=i$.

Now we have $a=b=i$, but this is a contradiction to our assumption $a \neq b$.  Therefore we have $\tau_a \tau_b \in \gamma_1$ if and only if $\tau_a \tau_b \in \gamma_2$.
\end{proof}
 
 We now show that the minimal free resolution of the square of the permutation ideal $ {\Tq}^2 $ is supported by $ \MM_q^2 $ and that these ideals are also Scarf ideals. To establish these results, we will use \cref{lem:scarf}.

\begin{theorem}[{\bf Main Theorem 2}]\label{scarfthm}
If $q$ is any positive integer, then the Scarf complex of the square of the permutation ideal, $\T_q^2$, coincides with the simplicial complex $ \MM_q^2 $.  That is, 
\[
\operatorname{Scarf}({\Tq}^2) = \MM_q^2.
\]
\end{theorem}

\begin{proof}  It has already been shown in \cref{mainthm} that for any monomial ideal $I$ minimally generated by $q$ monomials, the simplicial complex $\MM_q^2$ supports a free resolution of $I^2$. In particular, this implies that $\MM_q^2$ also supports a free resolution of ${\Tq}^2$.  Therefore, we have the following inclusion
 
\begin{equation}
\operatorname{Scarf}({\Tq}^2) \subseteq\MM_q^2.
\end{equation}
 Now, to prove the other inclusion, we show that every face $\gamma$ of  $\MM_q^2$ is a Scarf face. Using \cref{lem:scarf}, we need to show that for any $\gamma \in \MM_q^2$, we have the following two conditions satisfied:
\begin{enumerate}
    \item[C1:] $ m_{\gamma} \neq m_{\gamma \setminus \{v\}} $ for all vertices $ v \in \gamma $;
    \item[C2:] $ m_{\gamma \cup \{v\}} \neq m_{\gamma} $ for all vertices $ v \in V(\MM_q^2) \setminus \gamma $.
\end{enumerate}

Since $\gamma$ is a face of $\MM_q^2$, any subset of $\gamma$, and in particular $\gamma \setminus \{v\}$, will also be a face of $\MM_q^2$. This implies, by \cref{thmdistinctlabels}, that condition C1 is satisfied.

 Now, considering the second condition, if adding a vertex $ v \in V(\MM_q^2) \setminus \gamma $ to $\gamma$ results in $\gamma \cup \{v\}$ staying within $\MM_q^2$, then \cref{thmdistinctlabels} can once again be applied to conclude that condition C2 is satisfied. 
 
 Suppose $\gamma \cup \{v\}$ is no longer a face of $\MM_q^2$ and $m_{\gamma \cup \{v\}}=m_\gamma$. The face  $\gamma$ is contained within a unique facet, say $\M_i$, containing  ${\tau_i}^2$, and so we must have $v = {\tau_j}^2$ for some $j \neq i$.  We now consider a permutation $\sigma \in S_q$ such that $\sigma(j) = q$, so that   ${x_\sigma}^{2q} \mid m_{\gamma \cup \{v\}}=m_\gamma$. By \eqref{eq:sigma-gamma} we must have ${\tau_j}^2 \in \gamma$, a contradiction. Therefore, $m_{\gamma \cup \{v\}} \neq m_{\gamma}$. 
 
 We conclude that $\operatorname{Scarf}({\Tq}^2) = \MM_q^2$, and thus ${\Tq}^2$ is a Scarf ideal.

\end{proof}

 The following two corollaries follow directly from \cref{scarfthm} and highlight its significance.
 
\begin{corollary}\label{corominfree}
 The simplicial complex $\MM_q^2$ supports the minimal free resolution ${\Tq}^2$.
\end{corollary}

\begin{proof}
From \cref{mainthm}, it follows that the simplicial complex $\MM_q^2$ supports a free resolution of ${\Tq}^2$. Furthermore, by \cref{scarfthm}, we imply that this free resolution is, in fact, the Scarf resolution of ${\Tq}^2$. Therefore, the simplicial complex $\MM_q^2$ supports the minimal free resolution of ${\Tq}^2$.

\end{proof}

\begin{corollary}{\label{corobettivector}} 
Let $I$ be any monomial ideal minimally generated by $q$ monomials in the polynomial ring $S = K[x_1, x_2, \ldots, x_n]$. Then for all $0\leq i\leq {q \choose 2}$, we have
$$\beta_i({I}^2)\leq\beta_i({{\Tq}^2}).$$   
\end{corollary}

\begin{proof}
 By combining \cref{remarkmain}, \cref{scarfthm}, and \cref{corominfree}, we obtain the desired conclusion.

\end{proof}
We conclude this article with the following remark.
\begin{remark}\label{rem:lessvaribales} From \cref{iiilemma}, as well as the proofs of \cref{thmdistinctlabels} and \cref{scarfthm}, one can observe that the monomial generators of the permutation ideal $ \Tq $ can be replaced with alternative generators involving fewer variables, while still preserving the validity of both \cref{thmdistinctlabels} and \cref{scarfthm}. 

For instance, for each $ i \in [q] $, consider a monomial $ \tau_i' \in S_\T $ defined by
\[
\tau_i' = \prod\limits_{\sigma \in S_q \, : \, \sigma(i) \in \{q-1,\, q\}} x_{\sigma}^{\sigma(i)}.
\]
Define the corresponding new ideal by
\[
\Tq' = (\tau_1', \tau_2', \ldots, \tau_q').
\]
Then, \cref{scarfthm} continues to hold for the square of this ideal, ${(\Tq')}^{2}$. However, we prefer to work with the permutation ideal ${\Tq}$ rather than ${\Tq'}$. The reason is that, when studying higher powers of monomial ideals, it becomes necessary to introduce more variables to ensure distinct least common multiples in the lcm lattice of the ideal. However, to investigate higher powers, one must use alternative approaches such as techniques from Discrete Morse Theory. Our ongoing work  explores this direction in greater detail.

\end{remark}


\bibliographystyle{plain}
\bibliography{references}

\end{document}